 \documentclass[12pt]{amsart}

\usepackage{amssymb, amsmath, amsfonts}
\usepackage[raiselinks=false,colorlinks=true,citecolor=blue,urlcolor=blue,
linkcolor=blue,bookmarksopen=true,pdftex]{hyperref}

\usepackage{enumerate}
\usepackage[mathscr]{eucal}
\newtheorem{theorem}{Theorem}[section]

\newtheorem{lemma}[theorem]{Lemma}
\newtheorem{corollary}[theorem]{Corollary}

\newcommand{\homeo}{\textrm{Homeo}}

\newcommand{\so}{\textrm{SO}}

\newcommand{\wt}{\widetilde}
\newcommand{\im}{\textrm{Im}}

\newcommand{\ub}{\textrm{b}}
\newcommand{\bilip}{\textrm{Bilip}}
\newcommand{\diff}{\textrm{Diff}}
\newcommand{\pl}{\textrm{PL}}
\newcommand{\qi}{\mathcal{QI}}
\newcommand{\grad}{\textrm{grad}}
 
\newcommand{\fix}{\textrm{Fix}}
\begin{document}
\baselineskip=15.5pt
\title[Quasi-isometry groups of Euclidean spaces]{Embedding certain diffeomorphism groups in  the quasi-isometry groups of Euclidean spaces} 
\author{Oorna Mitra } 
\author{Parameswaran Sankaran}
\address{Institute of Mathematical Sciences, (HBNI), 
CIT Campus, Taramani, Chennai 600113.}

\email{oornamitra@imsc.res.in}
\email{sankaran@imsc.res.in}
\subjclass[2010]{20F65}
\keywords{Diffeomorphism groups, PL-homeomorphism groups, quasi-isometries of Euclidean spaces}
\thispagestyle{empty}
\date{}
\begin{abstract}
We show that certain groups of diffeomorphisms and PL-homeomorphisms embed in the group of all 
quasi-isometries of the Euclidean spaces. 
\end{abstract}

\maketitle

\section{Introduction} \label{intro}
Let $\Gamma$ be a finitely generated group with finite generating set $A
\subset \Gamma$.  The word metric $d_A$ makes $\Gamma$ into a metric space. Replacing $A$ 
by another finite generating set changes the metric on $\Gamma$ but not the quasi-isometry type of 
$(\Gamma, d_A)$.  Thus the quasi-isometric invariants of $(\Gamma,d_A)$ are intrinsic to the group $\Gamma$ itself.
The group of all self-quasi-isometries of $(\Gamma,d_A)$ is one such an invariant and is 
denoted $\qi(\Gamma)$.    More generally, for any metric space $X$, the group $\qi(X)$  
of all self-quasi-isometries of $X$ is an quasi-isometric invariant of $X$.

In general $\qi(\Gamma)$ is hard to determine.  It appears that there are very few families of groups $\Gamma$ 
for which $\qi(\Gamma)$ has an explicit description.  These include irreducible lattices in semisimple Lie groups, (see \cite{farb} and the references therein), solvable Baumslag-Solitar groups $BS(1,n)$ \cite[Theorem 7.1]{fm}, the groups $BS(m,n), 1<m<n$, 
\cite[Theorem 4.3]{whyte}, 
the group $B_2(\mathbb Z[1/m])$ of $2\times 2$-upper triangular matrices over $\mathbb Z[1/m]$, 
\cite{tw}, $B_n(\mathbb Z[1/p])$ for $p$ a prime and $n>2$, \cite{wortman},  
and,  the lamplighter groups 
$G\wr \mathbb{Z}$ with $G$ finite \cite{efw}.

Gromov and Pansu \cite[\S3.3.B]{gromov-pansu} noted that $\qi(\mathbb Z)$ is an infinite dimensional group.  
It was shown by Sankaran \cite{s} that $\qi(\mathbb Z)\cong \qi(\mathbb R)$ contains, for example, the 
free group of rank the continuum and a copy of the group of all compactly supported diffeomorphisms 
of $\mathbb R$.  The proof techniques used in \cite{s} heavily relied on the one-dimensionality of 
the real line.   In this note we show that, as in the case of $\mathbb Z$, the group $\qi(\mathbb Z^n)\cong \qi(\mathbb R^n)$ is 
large and contains many diffeomorphism groups.  More precisely,  we have the following result.  

\begin{theorem} \label{main} Let $n\ge 2$.  
The following groups can be embedded in $\qi(\mathbb{Z}^n)$. \\ 
{\em (i) $\bilip(\mathbb{S}^{n-1})$,} in particular, 
{\em $\diff^{\,r}(\mathbb{S}^{n-1}), 1\le r\le \infty,\pl (\mathbb S^{n-1}), $}\\
{\em (ii) $\diff^{\,r}(\mathbb D^n,\mathbb S^{n-1}), 1\le r\le \infty$}, where $\mathbb D^n$ denotes the disk 
$\{v\in \mathbb R^n\mid ||v||\le 1\}$,\\
{\em (iii)} $\qi(\mathbb R^k)\times \qi(\mathbb R^{n-k}), 1\le k<n$,\\
{\em (iv)  $\diff^{\,r}_\kappa(\mathbb R^n), 1\le r\le \infty, ~\pl_\kappa(\mathbb R^n)$.}\\
\end{theorem}

The notations used in the above theorem will be elaborated on in \S2. Part (ii) of the above theorem will be extended, in \S4,  
to the group of $C^r$-diffeomorphisms of the pair $(V,\partial V)$ where $V$ is any compact smooth $n$-dimensional 
manifold with boundary $\partial V$.   

Our proofs involve only elementary considerations. 
After observing that the induced Riemannian metric and the Euclidean metric  on the sphere are bi-Lipschitz equivalent, proof 
of part (i) involves radial extension of homemomorphisms of $\mathbb S^{n-1}$ to $\mathbb R^n$. 
Proof of part (ii) involves magnifying exponentially the features of 
homeomorphisms of the disk and replicating them on larger and larger pairwise disjoint disks positioned appropriately.  
Part (iv) is derived from (ii), while part (iii) is an elementary observation.

\section{Preliminaries} 
Let $(X,d)$ be a metric space. Recall that $f\in \homeo(X)$ is bi-Lipschitz if there exists a constant $\lambda\ge 1$ such that 
$(1/\lambda) d(x,y)\le d(f(x),f(y))\le \lambda d(x,y)$.  Any such $\lambda$ will be called a bi-Lipschitz constant for $f$.  
The group of all bi-Lipschitz homeomorphisms 
of $X$ is denoted by $\bilip(X,d)$.  Recall that two metrics $d$ and $\delta$ on $X$ are bi-Lipschitz equivalent if 
there exists a $\lambda\ge 1$ such that $(1/\lambda)d(x,y)\le \delta(x,y)\le \lambda d(x,y)$ for all $x,y\in X$.  If $d$ and $\delta$ are 
bi-Lipschitz equivalent, then $\bilip(X,d)=\bilip(X,\delta)$.  If $d$ is clear from the context, we shall abbreviate $\bilip(X,d)$ to $\bilip(X)$. 

We denote by $\diff^{\,r}(M)$ the group of all 
$C^r$-diffeomorphisms of a smooth manifold $M$ where $1\le r\le \infty$.  The group of compactly supported 
homeomorphisms of $M$ will be denoted $\homeo_\kappa(M)$.  If $G(M)$ is a group of 
homeomorphisms of $M$, the subgroup $\homeo_\kappa(M)\cap G(M)\subset G(M)$ will be 
denoted $G_\kappa(M)$.  If $M$ is a manifold with boundary, $\diff^{\,r}(M,\partial M)$ denotes the subgroup of $\diff^{\,r}(M)$ 
that fixes the boundary $\partial M$ pointwise.   
If $M$ is a Riemannian manifold, $\diff^{\,r}_\ub(M)$ denotes the subgroup of 
$\diff^{\,r}(M)$ consisting of those diffeomorphisms such that the norms $||T_x\phi||, ||T_x\phi^{-1}||$ of the 
differentials $T_x\phi: T_xM \to T_{\phi(x)}M$ are uniformly bounded in the following sense: 
There exists a real number $\lambda=\lambda(\phi)>1$ such that 
$\lambda^{-1}\le ||T_x\phi||, ||T_x\phi^{-1}||\le \lambda$ for all $x\in M$.  Note that $\diff_\ub^{\,r}(M)$ 
contains $\diff_\kappa^{\,r}(M)$, the group of all compactly supported diffeomorphisms of $M$. 

If $f:M\to M$ is a $C^r$-self-map of a Riemannian manifold, we define $||f||$ to be 
$||f||=\sup_{x\in M} ||T_xf||$ if it is finite, otherwise we set $||f||=\infty$. 
If both $||f||$ and $||f^{-1}||$ are finite, then, for any $x\in M$,  
$\inf_{x\in M} \inf_{||u||=1}||T_xf(u)||=1/||f^{-1}||>0$.  
It is easy to see that, when this happens, both $f$ and $f^{-1}$ are bi-Lipschitz.  
 Conversely, suppose that $f:M\to M$ is a bi-Lipschitz diffeomorphism. Then so is $f^{-1}$ and   
both $||f||, ||f^{-1}||$ are finite.   When $M$ is compact (with or without boundary), 
$||f||$ is always finite for all $C^r$-diffeomorphisms. 
An example of a diffeomorphism $f:M\to M$ with $||f||<\infty$ but $||f^{-1}||=\infty$ is 
$t\to t^3$ on $M=(0,1)$.   

Denote by $\mathbb R^n_0$ the punctured Euclidean space $\mathbb R^n\setminus \{0\}$.

In the case of $\mathbb R_{0}^n\subset \mathbb R^n$, the metric induced by the 
Riemannian metric is the {\it same} as the restriction of the Euclidean metric, which we shall use to define the 
group $\bilip(\mathbb R^n_0)$.  Note that if $f:\mathbb R^n_0\to \mathbb R^n_0$ 
is a bi-Lipschitz homeomorphism, then $f$ extends to a bi-Lipschitz homeomorphism of $\mathbb R^n$ fixing $0$.   
We shall identify $\bilip(\mathbb R^n_0)$ with the subgroup of 
$\bilip(\mathbb R^n)$ that fixes the origin.

\begin{lemma} \label{embeddingbilipsphere}
We keep the above notations. 
Let $\phi:\mathbb S^{n-1}\to \mathbb S^{n-1}$ be bi-Lipschitz.  
Then so is $\wt{\phi}: \mathbb R^n\to \mathbb R^n$, defined as 
$\wt{\phi}(v)=||v||\phi(v/||v||)~\forall v\in \mathbb R^n_0, \wt{\phi}(0)=0$. 
Moreover, $\phi\mapsto \wt{\phi}$ is a monomorphism {\em $\bilip(\mathbb S^{n-1})\to \bilip(\mathbb R^n)$} and 
$||\wt{\phi}-id||=\infty $ if $\phi$ is nontrivial.
\end{lemma}
\begin{proof}
Let $\lambda^{-1}||x-y||\le ||\phi(x)-\phi(y)||\le \lambda ||x-y||$ for all $x,y\in \mathbb S^{n-1}$.   Let $v, w\in \mathbb R^n_0$ 
where $||v||=||w||=:r$.  Then $||\wt{\phi}(v)-\wt{\phi}(w)||=||r\phi(v/r)-r\phi(w/r)||=r||\phi(v/r)-\phi(w/r)||\le r\lambda||v/r-w/r||
=\lambda||v-w||$. Similarly $||\wt{\phi}(v)-\wt{\phi}(w)||\ge \lambda^{-1}||v-w||$.  

Suppose that $||v||=s||w||, s> 1$. Set $v':=v/s$ so that $||\wt{\phi}(v')-\wt{\phi}(w)||\le \lambda ||v'-w||$. 
Note that $||\wt{\phi}(v)-\wt{\phi}(v')||=||v-v'||\le ||v-w||$ since $v,v'$ are on the same ray issuing from the origin; the last 
inequality holds because $v'$ is the point closest to $v$ on the sphere $S(0,||w||)$. 
Now  $||\wt{\phi}(v)-\wt{\phi}(w)||\le ||\wt{\phi}(v)-\wt{\phi}(v')||+||\wt{\phi}(v')-\wt{\phi}(w)||\le ||v-v'||+ \lambda||v'-w||
\le ||v-w||+\lambda||v-w||=(\lambda+1)||v-w||$, where the last inequality holds since in the triangle with vertices 
$v', v,w$, the angle at $v'$ is obtuse.  An entirely similar argument applies for $\wt{\phi^{-1}}=\wt{\phi}^{-1}$ and so 
$\wt{\phi}|_{\mathbb R^n_0}$ is bi-Lipschitz with bi-Lipschitz constant $(1+\lambda)$.  As observed 
already, this implies that $\wt{\phi}$ is bi-Lipschitz.  

Since $\wt{\phi}|_{\mathbb S^{n-1}}=\phi$ we see that $\phi\mapsto \wt{\phi}$ is a monomorphism.
As for the last statement, choose $x\in \mathbb S^{n-1}$ such that $\phi(x)\ne x$.  Then $||\wt{\phi}(rx)-rx||=r||\phi(x)-x||
\to \infty$ as $r\to \infty$ and so $||\wt{\phi}-id||=\infty$.
\end{proof}

We have the following corollary.  We denote by $\pl(\mathbb S^{n-1})$ the group of PL-homeomorphisms of 
the sphere with its standard PL-structure. 

\begin{corollary} \label{embeddiffsphere}
(i) The map $\phi\mapsto \wt{\phi}$ defines a monomorphism {\em $\diff^{\,r}(\mathbb S^{n-1})\to 
\diff^{\,r}_\ub(\mathbb R^n_0)$,} 
and, moreover, $||\wt{\phi}-id||=\infty$ if $\phi$ is non-trivial. \\
(ii)  The map $\phi\mapsto \wt{\phi}$ defines a monomorphism {\em $\pl(\mathbb S^{n-1})\to \bilip(\mathbb R^n)$} 
where $||\wt{\phi}-id||=\infty$ if $\phi$ is non-trivial.
\hfill $\Box$
\end{corollary}

\subsection{PL-homeomorphisms of $\mathbb R^n$}\label{plhomeomorphismsofRn}
Consider the standard triangulation $\mathcal T_0$ of $\mathbb R^n$ obtained by triangulating each unit cube with vertices 
in $\mathbb Z^n$.  A triangulation $\mathcal T$ of $\mathbb R^n$ is equivalent to $\mathcal T_0$ if there are 
subdivisions $\mathcal S_0$ and $\mathcal S$ of $\mathcal T_0$ and $\mathcal T$ respectively and a simplicial 
isomorphism $\mathcal S_0\to \mathcal S$.  A homeomorphism $f:\mathbb R^n\to \mathbb R^n$ is {\it piecewise linear} 
if there exists triangulations $\mathcal T, \mathcal T'$ which are equivalent to $\mathcal T_0$ and a simplicial 
isomorphism $\phi:(\mathbb R^n, \mathcal T)\to (\mathbb R^n, \mathcal T')$ that realises $f$, i.e., $f=|\phi|$. 

We denote by $\pl(\mathbb R^n)$ the group of all PL-homeomorphisms of $\mathbb R^n.$  Suppose that $f\in \pl(\mathbb R^n)$, realised by a simplicial isomorphism $\phi:(\mathbb R^n,\mathcal T)\to 
(\mathbb R^n,\mathcal T')$.  
Then $f$ is smooth at each interior point of any $n$-dimensional simplex of $\mathcal T$ 
since $f|\sigma$ is affine for any such simplex $\sigma$.  Denote by $T_\sigma f:\mathbb R^n\to\mathbb R^n$ 
the differential $T_pf$ at any interior point $p$ of $\sigma$.  (This depends only on $\sigma$ and not on the 
choice of $p$.)  We define $||f||$ to be $\sup ||T_\sigma f||\in \mathbb R\cup\{ \infty\}$ where the supremum is taken over all 
$n$-dimensional simplices of $\mathcal T$.  Then $||f||$ 
is independent of the simplicial isomorphism $\phi$ that represents $f$.   If $0<||f||<\infty$, then we have $0<||f^{-1}||<\infty$ 
and there exists a $\lambda\ge 1$ such that $\lambda^{-1}\le ||f||, ||f^{-1}||\le \lambda$. Moreover, there is such a 
$\lambda\ge 1$ if and only if $f, f^{-1}$ are bi-Lipschitz with bi-Lipschitz constant $\lambda$. 
We define $\pl_b (\mathbb R^n)$ 
to be the subgroup of bi-Lipschitz PL-homeomorphisms of $\mathbb R^n$.  The group $\pl_\kappa(\mathbb R^n)$ of all 
compactly supported PL-homeomorphisms is contained 
$\pl_b(\mathbb R^n)$.  

The unit (open) ball $\mathbb B^n:=B(0,1)$ has a natural PL-structure and one has a PL-homeomorphism $\pi:\mathbb R^n\to \mathbb B^n$. 
One may choose any diffeomorphism $\pi: \mathbb R^n\to \mathbb B^n$ and transport the PL-structure on $\mathbb R^n$ to 
$\mathbb B^n$.  
Then $\pi $ induces an isomorphism $\pl(\mathbb R^n)\to \pl(\mathbb B^n)$ defined as $f\mapsto \pi f\pi^{-1}$.  
Under this isomorphism $\pl_\kappa (\mathbb R^n)$ gets mapped onto $\pl_\kappa (\mathbb B^n)$.  On the other hand, 
any compactly supported (piecewise linear) homeomorphism of $\mathbb B^n$ extends uniquely to a (piecewise linear) homeomorphism 
of $\mathbb R^n$ supported in $\mathbb B^n$.  Thus $\pl_\kappa(\mathbb B^n)$ is naturally a subgroup of $\pl_\kappa(\mathbb R^n)$.

\subsection{Embeddings of $\bilip(\mathbb D^n,\mathbb S^{n-1})$ into $\bilip (\mathbb R^n)$}
For any integer $k\ge 0$, 
let $G_k=\homeo(D_k,\mathbb \partial D_k)$ where $D_k\subset \mathbb R^n$ is the closed unit disk centred at $ke_1$. 
Thus $D_0=\mathbb D^n$ and $G_0=\homeo(\mathbb D^n,\mathbb S^{n-1})$. The group 
$G_k$ will be regarded as a subgroup of $\homeo_\kappa(\mathbb R^n)$ whenever it is convenient to do so. 
Note that 
$\Phi_k:G_0\to G_k$, defined as $f\mapsto \tau_kf\tau_k^{-1}$ where $\tau_k$ is the translation $v\mapsto v+ke_1$ 
is an isomorphism of groups.  Also, if $h_j\in G_{2j}, j\ge 0$, then  $h_k\circ h_l=h_l\circ h_k$ whenever $k\ne l$ and 
the infinite composition $h_0\circ h_1\circ h_2 \circ \cdots $ is a well-defined homeomorphism of $\mathbb R^n$ 
whose support is contained in $\cup_{j\ge 0} D_{2j}$.  This element will be denoted more briefly as $\prod_{j\ge 0} h_j$.  
Explicitly \[\prod_{j\ge 0}h_j(v)= \left\{ \begin{array}{l l} 
h_j(v), & v\in D_{2j},\\
v, &v\notin \cup_{j\ge 0} D_{2j}.\\
\end{array}
\right. 
\]
  
One has an embedding $\Phi:G_0^\omega\cong \prod_{k\ge 0} G_{2k}\hookrightarrow \homeo(\mathbb R^n)$ defined as 
\[\Phi((g_j))= \prod_{j\ge 0} \Phi_{2j}(g_j).\]
for $(g_j)\in G_0^\omega$.

 Let $\delta:G_0\to G_0^\omega$ be the diagonal embedding.  
 We note that if $H\subset G_0$ is a group of bi-Lipschitz homeomorphisms,  
 then $\Phi(\delta(H))\subset \bilip(\mathbb R^n).$ 
 If $H\subset G_0\cap \diff^{\,r}_\kappa(\mathbb B^n)$, then $\Phi(\delta(H))\subset 
 \diff^{\,r}(\mathbb R^n)$.  Any {\it compactly supported} PL-homeomorphism of $\mathbb B^n$ extends to 
$\bilip(\mathbb D^n,\mathbb S^{n-1})$.  Thus $\pl_\kappa( \mathbb B^n)$ embeds in 
$G_0$.  It is clear that $\Phi(\delta(\pl_\kappa(\mathbb B^n)))\hookrightarrow \pl_b(\mathbb R^n)$.   

In general $\Phi((g_j))$ is not bi-Lipschitz even if $g_j\in G_j$ is bi-Lipschitz for every $j$.
However, if the $g_j, j\ge 0,$ are {\it uniformly bi-Lipschitz}, that is, if  there exists a $\lambda\ge 1$ such that $\lambda^{-1}||x-y||\le ||g_j(x)-g_j(y)||\le \lambda ||x-y|| ~\forall x,y\in \mathbb R^n$ for 
{\it every} $j\ge 0$, then $\Phi((g_j))$ is bi-Lipschitz.  
 
For any $(g_j)\in G_0^\omega$, the homeomorphism $\Phi((g_j))$ is quasi-isometrically equivalent to the identity since 
$\fix(\Phi((g_j)))$ is $1$-dense in $\mathbb R^n$.    
So we modify $\Phi:G_0^\omega\to \homeo(\mathbb R^n)$.  

Let $\rho_j:\mathbb R^n\to \mathbb R^n$ be defined as 
$\rho_j(v)=4^je_1+2^{j}v, j\ge 1$.   Then $\rho_j(\mathbb D^n)=D(4^je_1, 2^{j})=:C_j$.  It is convenient to set $C_0:=\mathbb D^n$ and $\rho_0=id_{\mathbb R^n}$.   
We note that $C_i\cap C_j=\emptyset$ if $i>j\ge 0$.   If $g\in G_0$, then $\rho_jg\rho_j^{-1}$ has support in $C_j$ 
and so $\rho_i g\rho_i^{-1}$ and $\rho_j g\rho_j^{-1}$ commute if $i\ne j$.  
Therefore $g\mapsto \prod_{j\ge 0}\rho_j
g\rho_j^{-1}$ is a well-defined homomorphism $\Psi:G_0\to \homeo (\mathbb R^n)$.  It is evident that 
$\Psi$ is a monomorphism since $\Psi(g)|_{\mathbb D^n}=g$.

\begin{lemma}\label{embeddingofG0}
With the above notations, let $g\in G_0$.  Then the following statements hold:\\ (i) if $g\in G_0$ is non-trivial, 
$||\Psi(g)-id||=\infty$. 
(ii) if {\em $g\in \bilip(\mathbb D^n,\mathbb S^{n-1})$,} then {\em $\Psi(g)\in \bilip(\mathbb R^n)$,}
(iii) if {\em $g\in \pl_\kappa(\mathbb B)\subset \homeo(\mathbb D,\mathbb S^{n-1})$,} then {\em $\Psi(g)\in \pl_b(\mathbb R^n)$.}

\end{lemma}
\begin{proof}
(i) Choose $x_0\in \mathbb B$ such that $g(x_0)\ne x_0$.  If $N\ge 1$ is any integer, choose an integer $k$ so large that 
$2^k||g(x_0)-x_0||>N$.  Then $||\Psi(g)(2^kx_0+4^ke_1)-2^kx_0-4^ke_1||=||\rho_kg(x_0)-2^kx_0-4^ke_1||
=||2^kg(x_0)-2^kx_0||>N$.  This proves (i).

(ii) Suppose that $\lambda \ge 1$ be a bi-Lipschitz constant for $g$.   We claim that 
$\lambda$ is also a bi-Lipschitz constant for $\psi(g)$.  To see this, first let $x,y\in D(4^je_1, 2^j)$.  Then $x_0:=
\rho_j^{-1}(x)=(x-4^je_1)/2^j\in \mathbb D$ and so  $g(x_0)\in \mathbb D$. Similarly $g(y_0)\in \mathbb D$ where 
$y_0:=\rho_j^{-1}(y)=(y-4^je_1)/2^j$. Therefore 
$\Psi(g)(x)-\psi(g)(y)=\rho_j g(x_0)-\rho_j g(y_0)=2^jg(x_0)-2^jg(y_0)=2^j(g(x_0)-g(y_0))$.  Since $\lambda^{-1}||x_0-y_0||
\le ||g(x_0)-g(y_0)||\le \lambda ||x_0-y_0||$, multiplying throughout by $2^j$ we see that 
$\lambda^{-1} ||x-y||=2^j\lambda^{-1} ||x_0-y_0||\le  ||\Psi(g)(x)-\Psi(g)(y)||\le \lambda 2^j ||x_0-y_0||=\lambda ||x-y||$.  

If $x_1, x_2\in \mathbb R^n$ are fixed by $\Psi(g)$, then, trivially $||\Psi(g)(x_1)-\Psi(g)(x_2)||=||x_0-y_0||$. 
Suppose that $x_0\in C_j, y_0\in C_k, j\ne k.$   The straight line segment joining $x_0, y_0$ meets $\partial C_j$ and 
$\partial C_k$ at unique points $x_1, x_2$ respectively and  we have $\psi(g)(x_i)=x_i, i=1,2$.
So,  $||\Psi(g)(x_0)-\Psi(g)(y_0)||\le ||\Psi(g)(x_0)-\Psi(g)(x_1)||+||\Psi(g)(x_1)-\Psi(g)(x_2)||+||\Psi(g)(x_2)
-\Psi(g)(y_0)||\le \lambda||x_0-x_1||+||x_1-x_2||+\lambda||x_2-y_0||\le \lambda(||x_0-x_1||
+||x_1-x_2||+||x_2-y_0||)=\lambda||x_0-y_0||$.  Similarly, 
$\lambda^{-1}||x_0-y_0||\le ||\Psi(g)(x_0)-\Psi(g)(y_0)||$.

\noindent
(iii)  Let $g\in \pl_\kappa(\mathbb B)\subset \pl_\kappa (\mathbb R^n)$.   
Since $\rho_j$ is affine, $\rho_j g\rho_j^{-1}$ is piecewise linear and it follows that $\Psi(g)$ is also 
piecewise linear.  As $g$ has compact support and is piecewise linear it is bi-Lipschitz. By (ii), it 
follows that $\Psi(g)$ is also bi-Lipschitz. 
\end{proof} 
 
\section{Proof of Theorem \ref{main}}

We are now ready to prove Theorem \ref{main}.   

\noindent
{\it Proof of Theorem \ref{main}}:  Since $\qi(\mathbb Z^n)\cong \qi(\mathbb R^n)$, we need only obtain embeddings into 
$\qi(\mathbb R^n)$.  

(i) One has a well-defined homomorphism $\eta:\bilip(\mathbb R^n)\to \qi(\mathbb R^n)$  defined 
as $f\mapsto [f]$.  The kernel of this homomorphism is the subgroup $\{f\in \bilip(\mathbb R^n)\mid ||f-id||<\infty\}$
By Lemma \ref{embeddingbilipsphere}, we have an embedding 
$\bilip(\mathbb S^{n-1})\to \bilip(\mathbb R^n)$ defined as $\phi \mapsto \wt{\phi}$ where $||\wt{\phi}-id||=\infty$ if $\phi\ne id$. 
Hence the restriction of $\eta$ to $\bilip(\mathbb S^{n-1})$ is a monomorphism.   Since elements of $\diff^{\, r}(\mathbb S^{n-1}), 
 1\le r \le \infty,$ and of $\pl(\mathbb S^{n-1})$ are bi-Lipschitz, it follows that $\diff^{\, r}(\mathbb S^{n-1})$ and $\pl(\mathbb S^{n-1}) $ are subgroups of $\bilip(\mathbb S^{n-1})$. This proves (i). 

(ii) By Lemma \ref{embeddingofG0}, one has an embedding $\Psi:\bilip(\mathbb D^n,\mathbb S^{n-1})\to \bilip(\mathbb R^n)$. 
Since $||\Psi(g)-id||=\infty$ for $g\ne id$, the composition $\eta\circ \Psi: \bilip(\mathbb D^n,\mathbb S^{n-1})\to 
\qi(\mathbb R^n)$ is a monomorphism.  
Since $\mathbb D^n$ is convex, the Riemannian metric on it induced from 
the Euclidean metric on $\mathbb R^n$ is the same as restriction of the Euclidean metric.   In particular the 
any $f\in \diff^{\, r}(\mathbb D^n,\mathbb S^{n-1})$ is bi-Lipschitz.  That is, $\diff^{\, r}(\mathbb D^n,\mathbb S^{n-1})$ is a subgroup 
of $\bilip(\mathbb D^n,\mathbb S^{n-1})$ and our assertion follows.

(iii) We regard $\mathbb R^n$ as $\mathbb R^k\times \mathbb R^l$ where $l=n-k$.  Any set-maps $f:\mathbb R^k\to \mathbb R^k$ 
and $g:\mathbb R^{l}\to \mathbb R^{l}$  
yield a set-map $f\times g:\mathbb R^n\to \mathbb R^n$.  We claim that if $f$ is a 
$(\lambda,\epsilon)$-quasi-isometry and if $g$ is $(\mu,\delta)$-quasi-isometry then so is $h:=f\times g$.  Since the Euclidean metric on $\mathbb R^n$ is bi-Lipschitz equivalent to the metric defined 
by the norm 
$||(x,y)||_1:=||x||+||y||$ where $x\in \mathbb R^k, y\in \mathbb R^{l}$, we need only show that $h$ is a quasi-isometry 
with respect to the metric induced by the $||\cdot||_1$ norm.  Routine verification shows that $h$ is a $(\nu, \epsilon+\delta)$-quasi-isometric 
embedding where $\nu:=\max\{\lambda, \mu\}$.  Also, 
if $\im(f)$ and $\im(g)$ are $C$-dense in $\mathbb R^k$ and $\mathbb R^l$ respectively,  
then $h=f\times g$ is $2C$-dense.   Thus if $f$ and $g$ are quasi-isometry equivalences 
of $\mathbb R^k$ and $\mathbb R^l$ respectively, so is $h$.  

Next, we show that if $f_0, f_1$ are quasi-isometry equivalences and if $||f_0- f_1||<\infty$, then the corresponding maps $h_0:=f_0\times g, h_1:=f_1\times g$ satisfy the condition $||h_0-h_1||<\infty$.   In fact we have 
$||f_0-f_1||=||h_0-h_1||$.  
To see this, let $(x,y)\in \mathbb R^n$.  Then $||h_0(x,y)-h_1(x,y)||_1=||(f_0(x),g(y))-(f_1(x),g(y)) ||_1=||f_0(x)-f_1(x)||$ and so 
$||h_0-h_1||=||f_0-f_1||$.  An entirely analogous argument shows that $f\times g_0\sim f\times g_1$ if 
$g_0\sim g_1$.  It follows that 
we have a well-defined homomorphism 
$\theta: \qi(\mathbb R^k)\times \qi(\mathbb R^l) \to\qi( \mathbb R^n)$ defined as $([f],[g])\mapsto [f\times g]$.  
Moreover, since $||(a,b)-(c,d)||_1=||a-c||+||b-d||$, we have $||f\times g-id_{\mathbb R^n}||=
||f-id_{\mathbb R^k}||+||g- id_{\mathbb R^l}||$.  Therefore 
if $||f-id_{\mathbb R^k}||=\infty$ or if  $||g-id_{\mathbb R^l}||=\infty$, then $||f\times g-id_{\mathbb R^n}||=\infty$.  It follows 
that $\theta$ is a monomorphism.

(iv) It was noted in \S\ref{plhomeomorphismsofRn} that 
$\pl_\kappa(\mathbb R^n)$ embeds in $\pl_\kappa(\mathbb B^n)\subset \bilip (\mathbb B^n)$.   
Using Lemma \ref{embeddingofG0} and arguing as in the proof of (i), we obtain 
that $\pl_\kappa(\mathbb R^n)$ embeds in $\qi(\mathbb R^n)$.  

It remains to consider the groups $\diff^{\, r}_\kappa(\mathbb R^n), 1\le r\le \infty.$   Since 
$\diff^{\, r}_\kappa(\mathbb R^n)\cong \diff^{\, r}_\kappa(\mathbb B^n)$ and since the latter group 
is naturally a subgroup of $\diff^{\, r}(\mathbb D^n,\mathbb S^{n-1})$ the 
assertion follows from (ii).     \hfill $\Box$

 \subsection{Embedding $\diff^{\, r}(V,\partial V)$ in $\qi(\mathbb Z^n)$}
Let $V\subset \mathbb R^n$ be a compact connected $n$-dimensional manifold with (non-empty) boundary 
$\partial V$.  We shall generalize  Theorem \ref{main} and show that $\bilip(V,\partial V)$ embeds in $\qi(\mathbb R^n)$.
First we need to compare two naturally defined metrics on $\partial V$.   
 
Let $M$ be a closed connected smooth submanifold of $\mathbb R^n$ of dimension $n-1$.  We denote the 
induced Riemannian metric on $M$ by $\delta$ and the metric inherited from the Euclidean metric 
on $\mathbb R^n$ by $d$; thus $d(x,y)=||x-y||$ whereas $\delta(x,y)$ is the induced length metric on $M$, i.e., 
$\delta(x,y)=\inf_\sigma l(\sigma)$ where $l(\sigma)$ is the length of a rectifiable arc 
$\sigma:[0,1]\to M$ from $x$ to $y$.      
Since $M$ is compact there is in fact a smooth geodesic $\gamma$ joining $x$ to $y$ so that $\delta(x,y)=l(\gamma)$.
We have the following lemma. 

\begin{lemma} \label{v} Let $M$ be a closed connected smooth $n-1$-dimensional submanifold of $\mathbb R^n$
With notations as above, $(M,\delta)$ and $(M,d)$ are bi-Lipschitz equivalent. 
\end{lemma}   
\begin{proof}
Let $x,y\in M$.  Then it is obvious that $\delta(x,y)\ge ||x-y||$.  It suffices to show 
the existence of a $\lambda\ge 1$ such that $\delta(x,y)\le \lambda ||x-y||$.  
Fix an $\epsilon>0$.   Suppose that $||x-y||\ge \epsilon$.  Let $D$ be the diameter of $(M,\delta)$.  Then, 
$\delta(x,y)\le D\le D ||x-y||/\epsilon$.  

Fix $x_0\in M$.  We shall show the existence of an $\epsilon=\epsilon(x_0)>0$ and a $\lambda_1=\lambda_1(x_0,\epsilon)\ge 1$ 
such that $\delta(x, y)\le \lambda_1||x-y||$ whenever $x,y\in B_d(x_0,\epsilon)$.

Choose an open neighbourhood $U\subset M$ of $x_0$ such that $U\subset V\times J\subset \mathbb R^n$ is the graph of a smooth function 
$f:V\to J$ where $V$ is an open ball in $\mathbb R^{n-1}$ of radius $r>0$, centred at $0$, and $J$ is an open interval containing the origin. 
We assume, as we may, that $f(0)=0$ so that $x_0\in M$ is the origin of $\mathbb R^n$.   
Let $a_0,a_1\in V$ and write $a(t)=(1-t)a_0+ta_1$.
Then $\sigma(t)
=(a(t), f(a(t))), 0\le t\le 1,$ is a smooth arc in $M$ joining $x=(a_0, f(a_0)), y=(a_1,f(a_1))\in M$.  We have  
$l(\sigma)=\int_{0}^1||\sigma'(t)||dt = \int_{0}^1\sqrt{||a'(t)||^2+f'(a(t))^2}dt=\int_0^1\sqrt{||a_1-a_0||^2+f'(a(t))^2}$. 
Since, by chain rule, $f'(a(t))=\grad(f)|_{a(t)}\cdot a'(t)=\grad f|_{a(t)}\cdot (a_1-a_0)$, we see that 
$|f'(a(t))|\le ||\grad(f)|_{a(t)}|| \cdot ||a_1-a_0||$.  
Let $c:=\sup_{v\in D(x_0,r/2)\cap V} ||\grad(f)|_{v}||$ and let $\lambda_1=\sqrt{1+c^2}$.  Set $\epsilon=\epsilon(x_0):=r/2$. 
Then, for any $x,y\in M\cap B(x_0,\epsilon)$,  
we have  $\delta(x,y)\le l(\sigma)\le \lambda_1 ||a_1-a_0||\le \lambda_1||x-y||$.  

Finally, by Lebesgue number lemma, there exists a number $\eta>0$ and {\it finitely} many open 
balls $B_d(x_1,\epsilon_1)\cap M,\ldots, B_d(x_k,\epsilon_k)\cap M$ which cover $M$  
such that (i) if $x,y\in M, ||x-y||<\eta$, then  $x,y\in B(x_j,\epsilon_j)$ for some $j\le k$, and, (ii) if $x,y\in B(x_j,\epsilon_j)$ 
then $\delta(x,y)\le \lambda_1(x_j,\epsilon)||x-y||$.  Set $K:=\max_{1\le j\le k} 
\lambda_1(x_j, \epsilon_j)$.  Then, for any $x,y\in M$ with $||x-y||<\eta,$ we have $\delta(x,y)\le K||x-y||$.  
Taking $\lambda= \max\{K, D/\eta\}$ we have $\delta(x,y)\le \lambda ||x-y||$ for 
all $x,y\in M$.  This completes the proof. 
\end{proof}

We have the following corollary.
\begin{corollary}\label{bilipmetricsonv}
Let $V$ be any compact connected smooth $n$-dimensional submanifold of $\mathbb R^n$.  Then 
induced Riemannian metric $\delta_V$ and Euclidean metric $d$ on $V$ are bi-Lipschitz equivalent. 
Consequently {\em $\bilip(V,\delta_V)=\bilip(V, d)$.}
\end{corollary}
\begin{proof}
Clearly $\delta_V(x,y)\ge ||x-y||$ for any $x,y\in V$.  In fact, equality holds if the straight line segment 
joining $x,y$ is entirely contained in $V$. In particular, if $x_0$ is an interior point of $V$, for some $\epsilon>0$ 
we have $ B_d(x_0,\epsilon)\subset V$ and $\delta_V(x_0,y)=||x_0-y||$ for all $y\in B_d(x_0, \epsilon)$.  

By the above lemma, there exists a $\lambda_1\ge 1$ such that for any component $M$ of $\partial V$, we have 
$\delta_M(x,y)\le  \lambda_1 ||x-y||~\forall x,y\in M$. (Here $\delta_M$ denotes the induced Riemannian metric 
on $M$.) If $x,y \in M$, a component  
of $\partial V$, then $\delta_V(x,y)\le \delta_M(x,y)\le \lambda_1 ||x-y||~\forall x,y\in M.$ 
 
Since $V$ is compact, there are only finitely many components of $\partial V$ and so there exists an $\eta>0$ such that 
if $x,y\in \partial V$ and $||x-y||<\eta$, then $x,y$ belong to the same component of $\partial V$.   
 
Let $x_0,y_0 \in V$. In general, the straight line segment joining $x_0,y_0$ can possibly meet $\partial V$ at infinitely many points.  
However, we can always find a finite partition of the line segment such that (i) the points of subdivision are in $\partial V$,
(ii) either the straight line segment joining successive points $x_i,x_{i+1}$ of subdivisions meet $V$ only in $\partial V$ or 
$||x_i-x_{i+1}||<\eta$.
More precisely, let 
$\sigma:[0,1]\to \mathbb R^n$ be defined as $t\mapsto (1-t) x_0+ty_0$.   
We choose a partition $t_0=0<t_1\le \cdots\le t_{k-1}<t_k=1$ 
of the interval $[0,1]$ satisfying the following two conditions.\\  (i) 
$x_j:=\sigma(t_j)$ belong to $\partial V, 0<j<k$. \\
(ii) Let $x_k:=y_0$ and let $0\le i<k$. Then either (a) the straight-line segment 
$[x_i,x_{i+1}]:=\sigma([t_i,t_{i+1}])$ meets $\partial V$ only at the end points,   
or, (b) $||x_{i+1}-x_i||<\eta$. 

If $[x_i,x_{i+1}]\subset V$, then 
$\delta_V(x_i,x_{i+1})=||x_i-x_{i+1}||$. If $[x_i, x_{i+1}]$ is not contained in $V$,  then $x_i,x_{i+1}$ belong 
to the {\it same} component---call it $M$---of $\partial V$. Therefore 
$\delta_V(x_i,x_{i+1})\le \delta_M(x_i,x_{i+1})\le \lambda_1||x_i-x_{i+1}||$. 

If $||x_i-x_{i+1}||<\eta$, 
then again $x_i,x_{i+1}\in\partial V$ are in the same component of $\partial V$ by our choice of $\eta$ and we conclude that 
$\delta_V(x_i,x_{i+1})\le \lambda_1||x_i-x_{i+1}||$.

Hence for any $x_0,y_0\in V$ we have  $\delta_V(x_0,y_0)\le \sum_{0\le j<k} \delta_V(x_{i+1},x_i)\le \sum \lambda_1||x_{i+1}-x_{i}||
=\lambda_1||x_{k}-x_{0}||=\lambda_1||x_0-y_0||.$  
\end{proof}

Let $V\subset \mathbb R^n$ be a smooth compact $n$-dimensional manifold with boundary.    
Choose $r>0$ so large that $V$ is contained in $B(0,r)$.  
Then $\diff^{\,r}(V,\partial V)$ 
naturally embeds in $\homeo(D(0,r))$ by extending each diffeomorphism to be identity outside $V$.    
Under this embedding, $\diff^{\,r}_\kappa(V\setminus \partial V)$ actually has image 
contained in $\diff_\kappa^{\,r}(B(0,r))$. In any case, as each element of $\diff^{\,r}(V,\partial V)$ bi-Lipschitz, 
so is its extension to $D(0,r)$. Thus $\diff^{\, r}(V,\partial V)$ is naturally a subgroup of $\bilip(D(0,r))$
We obtain the following as 
an immediate consequence of Lemma \ref{embeddingofG0} and Theorem \ref{main}(iii).
 
\begin{theorem} \label{vbdryv}Let $V\subset \mathbb R^n$ be a connected compact $n$-dimensional manifold with boundary.
There exists an embedding {\em $\Psi:\diff^{\,r}(V,\partial V)\to \bilip(\mathbb 
R^n)$} such that $||\Psi(f)-id||=\infty$ for any nontrivial element {\em $f\in \diff^{\,r}(V,\partial V)$.}  
Consequently $\diff^{\, r}(V,\partial V)$ embedds in $\qi(\mathbb Z^n)$.  \hfill  $\Box$ 
\end{theorem}

 \subsection{The spiral group}

Let $f: \mathbb R_{>0}\to \so(n)$ be a continuous map.  We write $f(t)=(f_{ij}(t))$.
Define $\phi=\phi_f: \mathbb R^n_0\to \mathbb R^n_0$, as $\phi(x)=f(||x||)) (x)$. 
It is readily verified that $\phi$ is a homeomorphism which maps each sphere centred at the origin to itself. It 
extends to a homeomorphism, again denoted $\phi_f$, of $\mathbb R^n$ that fixes the origin.  It is readily verified 
that $f\mapsto \phi_f$ is a homomorphism of groups $\textrm{Maps}(\mathbb R_{>0},\so(n))\to \homeo(\mathbb R^n)$ 
where the group structure on $\textrm{Maps}(\mathbb R_{>0}, \so(n))$ is obtained by pointwise operations. 
It is in fact a monomorphism.  We shall refer to the image of $\textrm{Maps}(\mathbb R_{>0},\so(n))$ as 
the {\it spiral group} of $\mathbb R^n$, and denote it by $\textrm{Spiral}(\mathbb R^n)$.

We will use the following norm for $n\times n$-matrices over $\mathbb R$.  If $A=(a_{ij})$, define $||A||_E:=(\sum a_{ij}^2)
^{1/2}$.  As usual, denoting the operator norm of $A$ by $||A||=\sup_{||x||=1}||Ax||$, we have 
$||A||\le ||A||_E\le \sqrt{n}||A||$.  Also $||Ax|| \le ||A||\cdot ||x||\le ||A||_E.||x||$.     

\begin{lemma}\label{rotations}
With notations as above, suppose that $f:\mathbb R_{>0}\to \so(n)$ is $C^1$ and 
satisfies the following condition:  there exist a constant $C=C(f)\ge 0$ such that, 
for all $1\le i,j\le n$, $|f'_{ij}(t)|\le C/t~\forall t>0.$  Then $\phi_f:\mathbb R^n
\to \mathbb R^n$ is bi-Lipschitz.
\end{lemma}
\begin{proof}
Let $x,y\in \mathbb R^n$.  Suppose that $r:=||x||=||y||$.  Then $||\phi(x)-\phi(y)||=||f(r)x-f(r)y||=||x-y||.$ 

Suppose that $s:=||y||>||x||=r$, $x\ne 0\ne y$.  Set $A:=f(r), B:=f(s)$.  Then 
$||\phi(x)-\phi(y)||=||Ax-By||\le ||Ax-Bx||+||Bx-By||\le ||A-B||.||x||+||x-y||$.

We have $B-A=(f_{ij}(s)-f_{ij}(r))=(f'_{ij}(t_{ij})).(s-r)$ for some $t_{ij}\in (r,s)$ by the mean value theorem.   By 
our hypothesis, $|f'_{ij}(t_{ij})|\le C/t_{ij}\le C/r$. 
So, using the norm $||X||_E=\sqrt{\sum_{i,j}|x_{ij}|^2}$ for 
$X=(x_{ij})\in M_n(\mathbb R)$, we obtain that 
$||B-A||_E^2= || (f'_{ij}(t_{ij}))||_E^2(s-r)^2 \le n^2C^2/t_{ij}^2\cdot(s-r)^2\le 
n^2 C^2(s-r)^2/r^2$.  This implies that 
$||Bx-Ax||\le nC(s-r)||x||/r=nC(||y||-||x||)\le nC||y-x||$.  So $||\phi(x)-\phi(y)||\le (nC+1)||x-y||~\forall x,y\in \mathbb R^n$.

Let $g(t)=f(t)^{-1}\in SO(n)$. Then $g_{ij}=f_{ji}, 1\le i,j\le n,$ and so $g$ also satisfies the condition $|g'_{ij}(t)|\le C/t ~\forall t>0$. 
Note that $\psi:=(\phi_f)^{-1}=\phi_g$ and so applying the above consideration 
to $\psi=\phi_g$ we obtain $||\psi(a)-\psi(b)||\le (nC+1)||a-b||$; equivalently $||x-y||/(nC+1)\le ||\phi(x)-\phi(y)||$ 
for all $x,y\in \mathbb R^n$. 
\end{proof}

It is readily checked that the space of all $C^1$-maps $f:\mathbb R_{>0}\to \so(n)$ satisfying the hypothesis of 
the above lemma 
forms a group under pointwise operations.  We shall denote this group by $\mathcal R$.  
Let 
$N:=\{f\in \mathcal R\mid \lim_{r\to \infty} f(r)=I_n\}$ and let    
$K=\{f\in \mathcal R\mid f(t)=I_n~\forall t\ge b,~\textrm{for some $b>0$}\}$.  Then $N$ and $K$ are subgroups of $\mathcal R$.  Indeed 
it can be seen that both $N$ and $K$ are normal subgroups of $\mathcal R$. 
In view of the above lemma, one has a well-defined map $\Phi: \mathcal R\to \qi(\mathbb R^n)$ defined as 
$f\mapsto [\phi_f]$.  We have the following 

\begin{theorem}  We keep the notations of Lemma \ref{rotations}.  Let $n\ge 2$.  
The map $\Phi: \mathcal R\to \qi(\mathbb R^n)$ sending $f$ to $[\phi_f]$ is a homomorphism whose kernel satisfies 
$K\subset \ker(\Phi)\subset N$. \\
\end{theorem}
\begin{proof} 
The verification that $\Phi$ is a homomorphism is routine.  It is evident 
that $K\subset \ker(\Phi)$.  
It remains to show that $\ker(\Phi)\subset N$.
Suppose that $\lim_{r\to \infty}f(r)\ne I$.  
We shall find a sequence $(x_k)_{k\ge 1}$ in $\mathbb R^n$ such that $\lim_{k\to \infty} ||\phi_f(x_k)-x_k||=\infty$. 
By the compactness of $\so(n)$,  there exists a monotone sequence $(r_k)$ of positive numbers 
such that $\lim_{k\to \infty} r_k=\infty$ and $\lim_{k\to \infty}f(r_k)=A\in \so(n)$, $A\ne I_n$.  Set 
$A_k:=f(r_{k})$.  

There exists a plane $V\in \mathbb R^n$ such that $A(V)=V$ and $A|_V$ is a rotation by angle $\theta, 0<\theta<2\pi$.  
Since $\lim_{k\to \infty} A_k=A$, there exist planes $V_k\subset \mathbb R^n$ and rotation angles $\theta_k\in [0,2\pi)$ 
such that $A_k|_{V_k}$ is rotation by $\theta_k$ and $\lim_{k\to \infty}\theta_k=\theta$.  Choose $\alpha <\pi$ 
such that $0<\alpha<\theta<2\pi-\alpha$. 
For $k$ sufficiently large and $x_k\in V_k$ with $||x_k||=r_k$, we have $\alpha <\theta_k<2\pi-\alpha$ and so $||\phi(x_k)-x_k||
=||A_kx_k-x_k||=2\sin (\theta_k/2) ||x_k||\ge 2\sin (\alpha/2) r_k$ which tends to $\infty$ as $k\to \infty$.           
\end{proof}

\subsection{Concluding remarks} 
We conclude this note with the following remarks and questions.    

(i)  Thurston \cite{thurston} showed that $\diff^{\infty}_\kappa(M)$ is simple for any connected smooth 
manifold.   Mather \cite{mather} showed that $\diff^{\,r+}_\kappa(M)$ is simple for $1\le r<\infty$; here $\diff^{\,r+}$ 
indicates the group of $C^r$-diffeomorphisms all whose $r$-th derivatives are Lipschitz.  Note that $\diff^{\,r+}(M)_\kappa\subset 
\diff^{r+1}_\kappa(M)$.   Filipkiewicz \cite{filipkiewicz} showed that the group $\diff^{\, r}_\kappa(M)$ determines the topology and 
smoothness structure of $M$.  See also Rybicki \cite{rybicki-pams}, \cite{rybicki}.  
These results, combined with Theorem \ref{vbdryv}, show that  
$\qi(\mathbb R^n)$ contains many pairwise non-isomorphic infinite dimensional simple groups.

(ii)  We do not know if $\bilip(\mathbb R^n)\to \qi(\mathbb R^n)$ is a surjective for all $n$. 
When $n=1$, this is left as an exercise in \cite[\S3.3.B]{gromov-pansu}; it also follows immediately 
from the main result of \cite{s}.     

We end this paper with the following two questions: \\  
\indent
(a) {\it Does there exist an isomorphism $\alpha: \qi(\mathbb R^n)\to \qi(\mathbb R^m)$ if $m<n$?}  \\
When $n\ge 2$, the group $\qi(\mathbb R^n)$ contains elements  of all possible orders since it contains $\textrm{SO}(n)$.   However, 
it is easily seen that the group $\qi(\mathbb R)$ contains no elements of order $k$ if $3\le k<\infty$.  This answers the question in the negative when $m=1$.  
 
(b)  {\it Does $\qi(\mathbb R^n)$ contain a group isomorphic to $\mathbb Z_2^k$ for $k>n$?}\\
Since $\textrm{O}(n,\mathbb R)$ naturally embeds in $\qi(\mathbb R^n)$, it is clear that  
$\mathbb Z_2^n$ also embeds in $\qi(\mathbb R^n)$. 
We remark that affirmative answer to (a) leads to an affirmative answer to (b).  
For, if $\alpha: \qi(\mathbb R^n)\to \qi(\mathbb R^m)$ is an isomorphism where $m<n$, then using Theorem \ref{main}(iii) 
we see that $\qi(\mathbb R^n)$ contains a group isomorphic to  $\qi(\mathbb R^n)\times \qi(\mathbb R^{n-m})$.  The latter group 
evidently contains an copy of $\mathbb Z^{2n-m}_2$.  
 Repeated application of the embedding of $\qi(\mathbb R^n)\times 
\qi(\mathbb R^{n-m})$ into $\qi(\mathbb R^n)$ leads to the conclusion that 
$\qi(\mathbb R^n)$ contains a copy of $\qi(\mathbb R^n)^k\subset \qi(\mathbb R^{n})^k\times \qi(\mathbb R^{n-m})$ for 
every $k\ge 1$.    

It seems plausible that the answers to both (a) and (b) are in the negative.

\noindent
{\bf Acknowledgments:} Research of both authors was partially supported by a XII Plan Project, Department of 
Atomic Energy, Government of India.

\end{document}